\documentclass{amsart}
              [1997/02/02 v1.2e PROC Author Class]

\copyrightinfo{2006}
  {American Mathematical Society}

\newtheorem{theorem}{Theorem}[section]
\newtheorem{corollary}[theorem]{Corollary}
\newtheorem{lemma}[theorem]{Lemma}
\newtheorem{proposition}[theorem]{Proposition}
\newtheorem{example}[theorem]{Example}
\newtheorem{question}[theorem]{Question}

\newtheorem{problem}[theorem]{Problem}

\begin{document}

\title{On questions which are connected with Talagrand problem}

\author{V.V.Mykhaylyuk}
\address{Department of Mathematics\\
Chernivtsi National University\\ str. Kotsjubyn'skogo 2,
Chernivtsi, 58012 Ukraine}
\email{vmykhaylyuk@ukr.net}


\commby{Ronald A. Fintushel}



\begin{abstract}
We prove the following results.

1. If $X$ is a $\alpha$-favourable space, $Y$ is a regular space, in which every separable closed set is compact, and $f:X\times
Y\to\mathbb R$ is a separately continuous everywhere jointly discontinuous function, then there exists a subspace $Y_0\subseteq
Y$ which is homeomorphic to $\beta\mathbb N$.

2. There exist a $\alpha$-favourable space $X$, a dense in $\beta\mathbb N\setminus\mathbb N$ countably compact space $Y$ and
a separately continuous everywhere jointly discontinuous function $f:X\times Y\to\mathbb R$.

Besides, it was obtained some conditions equivalent to the fact that the space $C_p(\beta\mathbb N\setminus\mathbb N,\{0,1\})$ of
all continuous functions $x:\beta\mathbb N\setminus\mathbb N\to\{0,1\}$ with the topology of point-wise convergence is a
Baire space.
\end{abstract}

\maketitle
\section{Introduction}

Investigation of joint continuity points set of separately continuous functions of two variables was started by R.~Baire in [1]. It was continued in papers of many mathematicians (H.~Hahn, W.~Serpinski, V.~Moran, I.~Namioka, M.~Talagrand, W.~Rudin, V.~Maslyuchenko and other; see, for example, [2] and the literature given there). I.~Namioka shows in [3] that for every strongly countably complete space $X$, compact space $Y$  and separately continuous function $f:X\times Y\to\mathbb R$ there exists a dense in $X$ $G_{\delta}$-set $A\subseteq X$ such that $f$ is jointly continuous at every point of set $A\times Y$. This result intensified the investigation of separately continuous functions defined on the product of Beaire and compact spaces. In particular, it was constructed in [4] an example of $\alpha$-favorable space $X$,
compact space $Y$ and separately continuous function $f:X\times Y\to\mathbb R$ such that the projection on $X$ of the set $D(f)$ of discontinuity points set of $f$ coincides with $X$. In this connection the following question was formulated in [4, Problem 3].

\begin{problem}\label{q:1.1}
Let $X$ be a Baire space, $Y$ be a compact space and $f:X\times Y\to\mathbb R$ be a separately continuous function. Is the function $f$ continuous at least at one point?
\end{problem}

It was shown in [5] that this question has the negative answer if the compactness of $Y$ to replace by $\tau$-compactness, where $\tau$ is an arbitrary infinite cardinal (a topological space $X$ is called {\it $\tau$-compact}, if every open cover of $X$ with the cardinality $\leq \tau$ has a finite subcover).

Note that for a completely regular space $Y$ and the space $X=C_p(Y,[0,1])$ of all continuous functions $x:Y\to[0,1]$ with the topology of pointwise convergence, or for a Hausdorff space $Y$ with a open-closed base and the space $X=C_p(Y,\{0,1\})$ of all continuous functions $x:Y\to\{0,1\}$ with the topology of pointwise convergence the separately continuous function $f:X\times Y\to\mathbb R$, $f(x,y)=x(y)$, is everywhere discontinuous. Therefore, it was naturally arises in the connection with Talagrand's Problem \ref{q:1.1} the question on investigation Baire property of spaces $C_p(Y,[0,1])$ and $C_p(Y,\{0,1\})$  for Hausdorff compact spaces $Y$.

In this paper we investigate the problem on the existence of everywhere discontinuous separately continuous function defined on the product of an $\alpha$-favorable space $X$ and a space $Y$, which satisfies a compactness-type conditions. Firstly we show that for an $\alpha$-favorable space $X$ and a regular space $Y$, in which every separable closed set is compact, the existence of an everywhere discontinuous function $f:X\times Y\to\mathbb R$, which quasicontinuous with respect to the first variable and continuous with respect to the second variable, imply the existence a subspace of $Y$ which is homeomorphic to Stone-Cech compactification $\beta\mathbb N$ of countable discrete space $\mathbb N$. Further, we construct an example od everywhere discontinuous separately continuous function defined on the product of an $\alpha$-favorable space $X$ and countably compact subspace $Y$ of space $\beta\mathbb N\setminus\mathbb N$. In the finish section we obtain some equivalent reformulations of the Baire property of the space of all continuous functions $x:\beta\mathbb N\setminus\mathbb N\to\{0,1\}$ with the topology of pointwise convergence.

\section{Everywhere discontinuous $KC$-functions}

Let $X$, $Y$, $Z$ be topological spaces and $f:X\times Y\to Z$. For every $x_0\in X$ and $y_0\in Y$ the mappings $f^{x_0}:Y\to Z$ ³ $f_{y_0}:X\to Z$ are defined by:
$$
f^{x_0}(y)=f(x_0,y)\,\,\,\,\,\,\,{\mbox
and}\,\,\,\,\,\,\,f_{y_0}(x)=f(x,y_0)
$$
for every $x\in X$ and $y\in Y$.

A mapping $f:X\to Y$ defined on a topological space $X$ and valued in a topological space $Y$ is called {\it quasicontinuous at a point $x_0\in X$}, if for every neighborhoods $U$ of $x_0$ in $X$ and $V$ of $f(x_0)$ in $Y$ there exists an open in $X$ nonempty set $U_1\subseteq U$ such that $f(U_1)\subseteq V$.
A mapping $f:X\to Y$ which is quasicontinuous at every point $x\in X$ is called {\it quasicontinuous}.

For topological spaces $X$, $Y$ and $Z$ the set of all mappings $f:X\times Y\to Z$ which is quasicontinuous with respect to the first variable and continuous with respect to the second variable we denote by $KC(X\times Y,Z)$.

\begin{lemma}\label{l:2.1} Let $X, Y, Z$ be topological spaces, $f\in KC(X\times Y, Z)$, $W_0, W_1$ open in $Z$ nonempty sets such that $\overline{f^{-1}(W_0)}=\overline{f^{-1}(W_1)}=X\times Y$. Then for every $n\in\mathbb N$, open in $X$ nonempty sets $G_1, G_2,\dots,G_n$ and reals $\theta_1, \theta_2,\dots,\theta_n\in\{0,1\}$ there exists $y_0\in Y$, open in $X$ nonempty sets $U_1, U_2,\dots,U_n$ such that $U_k\subseteq G_k$ ³ $f_{y_0}(U_k)\subseteq W_{\theta_k}$ for every $1\leq k\leq n$.
\end{lemma}

\begin{proof} Since all sets $f^{-1}(W_{\theta_k})$ are dense in $X\times Y$, for every $k\leq n$ the set $B_k=\{y\in Y: f(G_k\times\{y\})\cap W_{\theta_{k}}\ne \O\}$ is dense in $Y$. Moreover, the continuity of $f$ with respect to the second variable imply that all sets $B_k$ are open in $Y$. Therefore, the set $\bigcap\limits_{k=1}^{n}B_k$ is nonempty. We take $y_0\in\bigcap\limits_{k=1}^{n}B_k$. There exist points $x_k\in G_k$ for $k\leq n$ such that $f(x_k,y_0)\in W_{\theta_k}$. Now using the quasicontinuity of $f$ with respect to the first variable we found nonempty open in $X$ sets $U_k\subseteq G_k$ such that $f_{y_0}(U_k)\subseteq W_{\theta_k}$ for every $k\leq n$.
\end{proof}

Let $X$ be a topological space. Define the Shoquet game on $X$ in which two players $\alpha$ and $\beta$ participate. A nonempty open in $X$ set $U_0$
is the first move of $\beta$ and a nonempty open in $X$ set $V_1\subseteq U_0$ is the first move of $\alpha$. Further $\beta$ chooses a nonempty open in $X$
set $U_1\subseteq V_1$ and $\alpha$ chooses a nonempty open in $X$ set $V_2\subseteq U_1$ and so on. The player $\alpha$ wins if $\bigcap\limits_{n=1}^{\infty}V_n\ne\O$. Otherwise $\beta$ wins.

A topological space $X$ is called {\it $\alpha$-favorable} if $\alpha$ has a winning strategy in this game. A topological space $X$ is called {\it $\beta$-unfavorable} if $\beta$ has no winning strategy in this game. Clearly, any $\alpha$-favorable topological space $X$ is a $\beta$-unfavorable space. It was shown in [6] that a topological game $X$ is Baire if and only if $X$ is $\beta$-unfavorable.

Let $X$ be a topological space, $x_0\in X$, ${\mathcal U}$ be a system of all neighborhoods of $x_0$ in $X$ and $f:X\to \mathbb R$. The real
$$\omega_f(x_0) = \inf\limits_{U\in {\mathcal U}}\sup\limits_{x',x''\in U}|f(x')-f(x'')|$$
is called by {\it the oscillation of the function $f$ at the point $x_0$}.

\begin{theorem} \label{th:2.2} Let $X$ be an $\alpha$-favorable space, $Y$ be a Baire space and $f\in KC(X\times Y,\mathbb R)$ such that $D(f)=X\times Y$. Then there exists a sequence $(y_n)^{\infty}_{n=1}$ of points $y_n\in Y$ such that for every set $N\in\mathbb N$ there exists a continuous function $g:Y\to[0,1]$ such that $g(y_n)=1$, if $n\in N$, and $g(y_n)=0$, if $n\in\mathbb N\setminus N$.
\end{theorem}

\begin{proof} According to $[6]$, the space $X\times Y$ is Baire. Therefore there exist open in $X$ and $Y$ respectively sets $X_1\subseteq X$ and $Y_1\subseteq Y$, and  $\varepsilon>0$ such that $\omega_f(x,y)\geq\varepsilon$ for every $(x,y)\in X_1\times Y_1$. Using the fact that $X_1\times Y_1$ is Baire, we found nonempty open in $X$ and $Y$ respectively sets $X_0\subseteq X_1$ and $Y_0\subseteq Y_1$, reals $a,b \in\mathbb R$ with $a< b$ such that the sets $f^{-1}(W_0)$ end $f^{-1}(W_1)$ are dense in $X_0\times Y_0$, where $W_0=(-\infty,a)$ ³ $W_1=(b,+\infty)$.

Let ${\mathcal T}$ is the topology of the space $X$ and $\tau:\bigcup\limits^{\infty}_{n=1}{\mathcal T}^{2n+1}\to {\mathcal T}$ is a winning strategy of $\alpha$ in the Shoquet game on the topological space $X$.

For every $n\in\mathbb N\cup \{\omega_0\}$, $\xi=(\xi_1,\xi_2,\dots)\in \{0,1\}^n$ and $k<n$ we put $\xi|_k=(\xi_1,\xi_2,\dots,\xi_k)$.

Using the induction on $n\in\mathbb N$ we construct sequences of families $(U_{\xi}:\xi\in \{0,1\}^n)$ and $(V_{\xi}:\xi\in \{0,1\}^n)$ of open in $X$ nonempty sets $U_{\xi}$ and $V_{\xi}$ and a sequence $(y_n)_{n=1}^{\infty}$ of points $y_n\in Y$ such that:

$(i)$\,\,\, $V_{\xi}=\tau(U_{\xi|_1},V_{\xi|_1},\dots,U_{\xi})$ for every $n\in\mathbb N$ ³ $\xi\in\{0,1\}^n$;

$(ii)$\,\,\, $U_{\xi}\subseteq V_{\xi|_n}$ for every $n\in\mathbb N$ and $\xi\in\{0,1\}^{n+1}$;

$(iii)$\,\,\, $f_{y_n}(U_{\xi})\subseteq W_{\xi_{n}}$ for every $n\in\mathbb N$ ³ $\xi=(\xi_1,\xi_2,\dots,\xi_n)\in \{0,1\}^n$.

According to Lemma \ref{l:2.1}, we choose a point $y_1\in Y_0$ and open in $X$ nonempty sets $U_0$ and $U_1$ such that $f_{y_1}(U_\xi)\subseteq W_\xi$ for every $\xi\in\{0,1\}$. Put $V_0=\tau(U_0)$ ³ $V_1=\tau(U_1)$.

Assume that the points $y_k\in Y$, the families $(U_\xi:\xi\in \{0,1\}^k)$ and $(V_\xi:\xi\in \{0,1\}^k)$ for $k\leq n$ are constructed. For every $\xi=(\xi_1,\xi_2,\dots,\xi_{n+1})\in \{0,1\}^{n+1}$ put $G_\xi=V_{\xi|_n}$ ³ $\theta_\xi=\xi_{n+1}$. Then according to Lemma \ref{l:2.1}, there exist $y_{n+1}\in Y$ and a family $(U_\xi:\xi\in \{0,1\}^{n+1})$ of nonempty open in $X$ sets $U_\xi$ such that $U_\xi\subseteq G_\xi$ and $f_{y_{n+1}}(U_\xi)\subseteq W_{\theta_\xi}$, that is the conditions $(ii)$ and $(iii)$ are true for every $\xi\in \{0,1\}^{n+1}$. It remains to put $V_\xi=\tau(U_{\xi|_1},V_{\xi|_1}, \dots, U_\xi)$ for all $\xi\in
\{0,1\}^{n+1}$.

Show that the sequence $(y_n)^{\infty}_{n=1}$ is the required. Let $N\subseteq \mathbb N$. Put $\xi_n=1$, if $n\in N$, $\xi_n=0$, if $n\in\mathbb N\setminus N$, and $\xi=(\xi_n)_{n=1}^{\infty}$. According to $(i)$ and $(ii)$, we have $U_{\xi|_{n+1}}\subseteq V_{\xi|_n}\subseteq U_{\xi|_n}$ for every $n\in\mathbb N$. Note that the player $\alpha$ plays accordingly with the winner strategy $\tau$ in the Shoquet game
$$U_{\xi|_1}\subseteq V_{\xi|_1}\subseteq \dots .$$ Therefore $\bigcap\limits_{n=1}^{\infty}U_{\xi|_n}\ne\O$.

Let $x_0\in \bigcap\limits_{n=1}^{\infty}U_{\xi|_n}$. According to $(iii)$, we have $f(x_0,y_n)\in W_1$, if $n\in N$, and $f(x_0,y_n)\in W_0$, ÿêùî $n\in\mathbb N\setminus N$. Take an continuous function $\varphi:\mathbb R\to [0,1]$ such that $W_0\subseteq \varphi^{-1}(0)$ and $W_1\subseteq \varphi^{-1}(1)$.
Then the continuous function $g:Y\to [0,1]$, $g(y)=\varphi(f(x_0,y))$, is the required.
\end{proof}

The following Corollary is a main result of this section.

\begin{corollary}\label{c:2.3} Let $X$ be an $\alpha$-favorable space, $Y$ be a regular space in which every separable closed set is compact and $f\in KC(X\times Y,\mathbb R)$ such that $D(f)=X\times Y$. Then there exists a compact in $Y$ set $Y_0$, which is homeomorphic to the space $\beta\mathbb N$.
\end{corollary}

\begin{proof} It easy to see that every regular space, in which each separable closed set is compact, is $\alpha$-favorable, in particular, a Baire space. According to Theorem \ref{th:2.2}, we choose a sequence $(y_n)^{\infty}_{n=1}$ which satisfies the corresponding condition and put $Y_0=\overline{\{y_n:n\in\mathbb N\}}$. Then according to [7, Corollary 3.6.4] the space $Y_0$ is homeomorphic to $\beta\mathbb N$.
\end{proof}

\section{Stone-Cech compactification and $p$-sets}

A system ${\mathcal A}$ of subsets of a set $X$ is called {\it ultrafilter on $X$}, if the following conditions hold:

$(a)$\,\,\,$\bigcap {\mathcal B}\ne\O$ for every finite system ${\mathcal B}\subseteq{\mathcal A}$;

$(b)$\,\,\,euther $A\in {\mathcal A}$ or $X\setminus A\in {\mathcal A}$ for every set $A\subseteq X$.

Let ${\mathcal F}$ be the collection of all ultrafilters on $\mathbb N$. Clearly (see [7,Corollary 3.6.4]) that a mapping $\varphi: \beta\mathbb N\to {\mathcal F}$,
$\varphi(x)=\{A\subseteq\mathbb N: x\in\overline{A}\}$, is a bijection, besides $\varphi(n)=\{A\subseteq \mathbb N: n\in A\}$ for every $n\in\mathbb N$. Moreover, for every $x\in \beta\mathbb N\setminus\mathbb N$ the ultrafilter $\varphi(x)$ is called {\it nontrivial} and it has the following property: if $A\in \varphi(x)$ and $B\subseteq\mathbb N$ such that $|A\setminus B|<\aleph_0$ then $B\in\varphi(x)$.

Further, the elements $x\in \beta\mathbb N\setminus \mathbb N$ we will identify with $\varphi(x)$. Note that for every closed-open nonempty set $U\subseteq\beta\mathbb
N\setminus\mathbb N$ there exists an infinite set $A\subseteq \mathbb N$ such that $U=\{x\in\beta\mathbb N\setminus\mathbb N: A\in x\}$.

\begin{lemma}\label{l:3.1} Let $X=\beta\mathbb N\setminus \mathbb N$, $(A_n)^{\infty}_{n=1}$ ³ $(B_n)^{\infty}_{n=1}$ be sequences of closed in $X$ sets $A_n, B_n\subseteq X$ such that $\overline{A}\cap B=A\cap \overline{B}=\O$, where $A=\bigcup\limits_{n=1}^{\infty}A_n$ and $B=\bigcup\limits_{n=1}^{\infty}B_n$. Then
$\overline{A}\cap\overline{B}=\O$.
\end{lemma}

\begin{proof} Using the induction on $n$ it easy to construct sequences $(U_n)_{n=1}^{\infty}$ and $(V_n)_{n=1}^{\infty}$ of closed-open in $X$ sets $U_n$ and $V_n$ such that $A_n\subseteq U_n$, $B_n\subseteq V_n$ for every $n\in\mathbb N$ and $(\bigcup\limits_{n=1}^{\infty}U_n)\bigcap (\bigcup\limits_{n=1}^{\infty}V_n)=\O$. We choose sequences $(S_n)_{n=1}^{\infty}$ and $(T_n)_{n=1}^{\infty}$ of sets $S_n,T_n\subseteq \mathbb N$ such that $U_n=\{x\in X: S_n\in x\}$ and $V_n=\{x\in X: T_n\in x\}$ for every $n\in\mathbb N$. Since $U_n\bigcap V_m=\O$, $|S_n\bigcap T_m|<\aleph_0$ for every $n,m\in\mathbb N$. Put $$S=\bigcup\limits_{n=1}^{\infty}(S_n\setminus(\bigcup\limits_{k=1}^{n}T_k))\,\,\,{\mbox and}\,\,\,T=\bigcup\limits_{n=1}^{\infty}(T_n\setminus(\bigcup\limits_{k=1}^{n}S_k)).$$

We show that $S\bigcap T=\O$. Suppose that $m\in S\bigcap T$. Taking onto account that $S\subseteq\bigcup\limits_{n=1}^{\infty}S_n$ and $T\subseteq\bigcup\limits_{n=1}^{\infty}T_n$, we put $i={\rm min}\{n\in\mathbb N:m\in S_n\}$ ³ $j={\rm min}\{n\in\mathbb N:m\in T_n\}$.

If $i\leq j$, then $m\not\in T_n$ for $n<j$ and $m\not\in T_n\setminus(\bigcup\limits_{k=1}^{n}S_k)$ for $n\geq j$. Thus, $m\not\in T$, a contradiction. Analogously, $m\not\in S$ if $j\leq i$.

Morefore, note that $S_n\setminus S\subseteq S_n\setminus(S_n\setminus \bigcup\limits_{k=1}^{n}T_k)\subseteq \bigcup\limits_{k=1}^{n}(S_n\bigcap T_k)$ and $T_n\setminus
T\subseteq T_n\setminus(T_n\setminus \bigcup\limits_{k=1}^{n}S_k)\subseteq \bigcup\limits_{k=1}^{n}(T_n\bigcap S_k)$ for every $n\in\mathbb N$. Therefore all sets  $S_n\setminus S$ and $T_n\setminus T$ are finite, $U_n\subseteq U=\{x\in X: S\in x\}$ and $V_n\subseteq V=\{x\in X: T\in x\}$ for every $n\in\mathbb N$, besides the closed-open in $X$ sets $U$ and $V$ such that $U\bigcap V=\O$.
\end{proof}
The next result follows from [7, Corollary 3.6.4].

\begin{corollary} \label{c:3.2} Let  $A\subseteq \beta \mathbb N\setminus \mathbb N$ be a countable set. Then the closure $\overline{A}$ of $A$ in the space $\beta \mathbb N\setminus \mathbb N$ is homeomorphic to the Stone-Cech compactification of the space $A$.
\end{corollary}

A subset $A$ of a topological space $X$ is called {\it $p$-set}, if
$$A\subseteq G={\rm
int}\left(\bigcap\limits_{n=1}^{\infty}G_n\right)$$
for every sequence $(G_n)_{n=1}^{\infty}$ of open in $X$ sets $G_n$ with $A\subseteq G_n$ for every $n\in\mathbb N$.

\begin{proposition}\label{p:3.3}
Let ${\mathcal P}$ be a system of all closed nowhere dense $p$-sets in $X=\beta \mathbb N\setminus \mathbb N$. Then

$(i)$\,\,\, the set $\bigcup{\mathcal P}=\bigcup\limits_{P\in{\mathcal P}}P$ is dense in $X$;

$(ii)$\,\,\, $U\bigcap P\in{\mathcal P}$ for every closed-open in $X$ sets $U$ and $P\in{\mathcal P}$;

$(iii)$\,\,\,$P=\overline{\bigcup\limits_{n=1}^{\infty}P_n}\in{\mathcal P}$ for every sequence $(P_n)_{n=1}^{\infty}$ of sets $P_n\in{\mathcal P}$.
\end{proposition}

\begin{proof} Conditions $(i)$ and $(ii)$ immediately follows from [7, exercise 3.6.À]. We prove $(iii)$. Let  $P_n\in{\mathcal P}$ for every $n\in\mathbb N$. Since every nonempty $G_{\delta}$-set in $X$ has nonempty interior (see [7, exercise 3.6.À]), the set $P=\overline{\bigcup\limits_{n=1}^{\infty}P_n}$ is nowhere dense in $X$. It remains to show that $P$ is a $p$-set in $X$.

Let $(G_n)^{\infty}_{n=1}$ be a sequence of open in $X$ sets $G_n$ such that $P\subseteq\bigcap\limits_{n=1}^{\infty}G_n$. Put $A_n=X\setminus G_n$ for every $n\in\mathbb N$, $A=\bigcup\limits_{n=1}^{\infty}A_n$ and $B=\bigcup\limits_{n=1}^{\infty}P_n$. Since $P_n\in{\mathcal P}$ for every $n\in\mathbb N$, $B\subseteq {\rm int}(\bigcap\limits_{n=1}^{\infty}G_n)$, that is $B\bigcap \overline{A}=\O$. Moreover, $P=\overline{B}\subseteq \bigcap\limits_{n=1}^{\infty}G_n$, therefore $\overline{B}\bigcap A=\O$. According to Lemma \ref{l:3.1},  we have $\overline{A}\bigcap\overline{B}=\O$, that is $P\subseteq{\rm
int}(\bigcap\limits_{n=1}^{\infty}G_n)$.
\end{proof}

Now we give an example of everywhere discontinuous separately continuous function defined on the product of $\alpha$-favorable space $X$ and counably compact dense subspace of $\beta\mathbb N\setminus\mathbb N$.

\begin{example}\label{ex:3.4} Let $X$ be a set of all continuous functions $x:\beta \mathbb N\setminus \mathbb N \to \{0,1\}$, ${\mathcal P}$ be a system of all closed nowhere dense $p$-sets $P\subseteq \beta \mathbb N\setminus \mathbb N$ and $Y=\bigcup\limits_{P\in{\mathcal P}}P$. We consider the space $X$ with the topology of uniform convergence on sets of the system ${\mathcal P}$. That is for every $x\in X$ the system $\{U(x,P):P\in{\mathcal P}\}$ forming a base of neighborhoods of $x$ in the space $X$, where $U(x,P)=\{x'\in X: x'(t)=x(t)\,\,\forall t\in P\}$.

Consider the separately continuous function $f:X\times Y\to\mathbb R$, $f(x,y)=x(y)$. Since in $Y$ the system of all closed-open sets forming a base of the topology and every set $P\in{\mathcal P}$ is nowhere dense in $Y$, the function $f$ is discontinuous at every point $(x_0,y_0)\in X\times Y$.

Now we show that the space $X$ is $\alpha$-favorable. Let $(U_n)^{\infty}_{n=1}$ is a decreasing sequence of nonempty basic open sets in $X$. Then there exist increasing sequences $(P_n)^{\infty}_{n=1}$ and $(Q_n)^{\infty}_{n=1}$ of sets $P_n, Q_n\in{\mathcal P}$ such that
$$U_n=\{x\in X: x(y)=0\,\,\forall y\in P_n \,\,{\mbox and}\,\, x(y)=1 \,\,\forall y\in Q_n\}.$$
Put $P=\overline{\bigcup\limits_{n=1}^{\infty} P_n}$ and $Q=\overline{\bigcup\limits_{n=1}^{\infty} Q_n}$. Proposition \ref{p:3.3} imply that $P,Q\in{\mathcal P}$. Moreover, it follows from the definition of $p$-set that $P_n\bigcap Q = P\bigcap Q_n =\O$ for every $n\in\mathbb N$. Therefore according to Lemma \ref{l:3.1}, $P\bigcap Q = \O$. Now choose a continuous on $\beta \mathbb N\setminus\mathbb N$ function $x_0$ such that $x_0(y)=0$ for every $y\in P$ and $x_0(y)=1$ for every $y\in Q$ and obtain $x_0\in \bigcap\limits_{n=1}^{\infty}U_n$.
\end{example}

A positive answer to the following question gives the solution of Talagrand problem.

\begin{question}\label{q:3.5} Is there equality $\beta \mathbb N\setminus\mathbb N= \bigcup\limits_{P\in{\mathcal P}}P$, where ${\mathcal
P}$ is the system of all closed nowhere dense $p$-sets in $\beta \mathbb N\setminus\mathbb N$?
\end{question}

\section{Some property of $C_p(\beta\mathbb N\setminus\mathbb N, \{0,1\})$}

Let $X$ be topological space and $(A_n)_{n=1}^{\infty}$ be a sequence of sets $A_n\subseteq X$. We say that {\it the sequence $(A_n)_{n=1}^{\infty}$ weakly converges to 
$x_0\in X$ in $X$}, if for every neighborhood $U$ of $x_0$ in $X$ there exists an integer $n_0\in\mathbb N$ such that $U\bigcap A_n\ne\O$ for every $n\geq n_0$.

\begin{theorem} \label{th:4.1} Let $Y=\beta \mathbb N\setminus \mathbb N$ and $X=C_p(Y,\{0,1\})$. Then the following conditions are equivalent:

$(i)$ $X$ is meagre;

$(ii)$ $X$ is not Baire;

$(iii)$ there exists a sequence $(E_n)^{\infty}_{n=1}$ of finite pairwise disjoint sets $E_n\subseteq Y$ which weakly converges to a point $y_0\in Y$;

$(iv)$ there exists a sequence $(E_n)_{n=1}^\infty$ of finite pairwise disjoint sets $E_n\subseteq Y$ which weakly converges to every point $y\in \bigcup\limits_{n=1}^\infty E_n$.
\end{theorem}

\begin{proof} For every disjoint sets $A,B\subseteq Y$ we put 
$$U(A,B)=\{x\in X:x(a)=0\,\,\forall a\in A, x(b)=1\,\,\forall b\in B\}.$$
Clearly that the system 
$$\{U(A,B): A,B\subseteq Y \mbox{are\,\,finite\,\,and\,\,disjoint}\}$$ forming a base of the topology of $X$.

The implications $(i)\Rightarrow (ii)$ and $(iv)\Rightarrow (iii)$ are obvious.

$(ii)\Rightarrow (iii)$. Let $A_0,B_0\subseteq Y$ are finite disjoint sets such that $X_0=U(A_0,B_0)$ is meagre in $X$, that is $X_0=\bigcup\limits_{n=1}^\infty X_n$, where $(X_n)_{n=1}^{\infty}$ is a increasing sequence of nowhere dense in $X$ sets.

\begin{quotation} 
\begin{lemma}\label{l:4.2} For every $n\in\mathbb N$ and finite set $C\subseteq Y$ there exist finite disjoint sets $A,B\subseteq Y\setminus C$ such that $U(A,B)\cap X_n=\O$.
\end{lemma}

\begin{proof} Let $D=C\setminus (A_0\cup B_0)=\{d_1,\dots,d_m\}$, moreover without loss of generality we can propose that $m\ge 1$. Let $D_1,\dots,D_{2^m}$ are all subsets of set $D$. We put $C_k=D\setminus D_k$ for $k=1,\dots,2^m$.

Show that $X_0=\bigcup\limits_{k=1}^{2^m}U(A_0\cup C_k,B_0\cup D_k)$. Since $U(A_0\cup C_k,B_0\cup D_k)\subseteq X_0$ for every $k=1,\dots,2^m$, $\bigcup\limits_{k=1}^{2^m}U(A_0\cup C_k,B_0\cup D_k)\subseteq X_0$.

Let $x\in X_0$. Using $k\in\{1,\dots,2^m\}$ such that $C_k=\{y\in D: x(y)=0\}$ ³ $D_k=\{y\in D:x(y)=1\}$ we obtain that $x\in U(A_0\cup C_k,B_0\cup D_k)$.

Since $X_n$ is meagre in $X$, there exist finite disjoint sets $S_1,T_1\subseteq Y\setminus (A_0\cup B_0\cup D)$ such that $U(A_0\cup C_1\cup S_1,B_0\cup D_1\cup T_1)\cap X_n=\O$. Further, using the fact that $X_n$ is meagre in $X$ by the induction on $k$ we construct sequences $(S_k)_{k=1}^{2^m}$ and $(T_k)_{k=1}^{2^m}$ of pairwise disjoint sets $S_k,T_k\subseteq Y$ such that $\left(S_k\cup T_k\right)\cap \left(\bigcup\limits_{i=1}^{k-1}(S_i\cup T_i)\cup A_0\cup B_0\cup D\right)=\O$ and $U\left(\bigcup\limits_{i=1}^kS_i\cup A_0\cup C_k,\bigcup\limits_{i=1}^kT_i\cup B_0\cup D_k\right)\cap X_n=\O$ for every $k\in\{1,\dots,2^m\}$.

We put $A=\bigcup\limits_{k=1}^{2^m}S_k$ and $B=\bigcup\limits_{k=1}^{2^m} T_k$. Show that $U(A,B)\cap X_n=\O$. Assume that $x\in U(A,B)\cap X_n$. Since $X_n\subseteq X_0$, there exists $k\in\{1,\dots,2^m\}$ such that $x\in U(A_0\cup C_k,B_0\cup D_k)$. Then $x\in U(A\cup A_0\cup C_k,B\cup B_0\cup D_k)\cap X_n\subseteq U\left(\bigcup\limits_{i=1}^kS_i\cup A_0\cup C_k,\bigcup\limits_{i=1}^kT_i\cup B_0\cup D_k\right)\cap X_n$. But this contradicts to the choice of sets $S_k$ and
$T_k$.
\end{proof}
\end{quotation}

It follows from Lemma \ref{l:4.2} that there exist sequences $(A_n)_{n=1}^\infty$ and $(B_n)_{n=1}^\infty$ of finite disjoint sets $A_n,B_n\subseteq Y$ such that $(A_n\cup B_n)\cap \left(\bigcup\limits_{k=0}^{n-1} (A_k\cup B_k)\right)=\O$ and $U(A_n,B_n)\cap X_n=\O$ for every $n\in \mathbb N$.

Suppose that $(iii)$ is false. We consider the sequence $(E_n)_{n=1}^\infty$ of pairwise disjoint sets $E_n=A_n\cup B_n$. Using the denial of $(³³³)$ and the finiteness of $E_0=A_0\cup B_0$ we found a finite set $N_1\subseteq {\mathbb N}$ such that $E_0\cap\overline{\bigcup\limits_{n\in N_1}E_n}=\O$.

Using similar reasoning with respect to the set $E_{n_1}$, where $n_1=\min N_1$, the sequence $(E_n)_{n\in N_1}$, we choose an infinite set $N_2\subseteq N_1$ such that  $E_{n_1}\cap \overline{\bigcup\limits_{n\in N_2}E_n}=\O$. Continuing this process to infinity we obtain a strictly decreasing sequence $(N_k)_{k=1}^\infty$ of infinite sets $N_k\subseteq {\mathbb N}$ such that 
$$
E_{n_{k-1}}\cap \overline{\bigcup\limits_{n\in N_k}E_n}=\O,
$$
for every $k\in\mathbb N$, where $n_k=\min N_k$ and $n_0=0$.

Put $\tilde{A}_k=A_{n_{k-1}}$, $\tilde{B}_k=B_{n_{k-1}}$ for every $k\in\mathbb N$, $A=\bigcup\limits_{k=1}^\infty \tilde{A}_k$ and $B=\bigcup\limits_{k=1}^\infty \tilde{B}_k$. According to the choice of $(n_k)^{\infty}_{k=1}$ we have $E_{n_k}\bigcap\left(\overline{\bigcup\limits_{i\ne k}E_{n_i}}\right)=\O$ for every $k\in\mathbb N$. Therefore $\tilde{A_k}\cap\overline{B}=\overline{A}\cap \tilde{B_k}=\O$ for every $k\in\mathbb N$ and $\overline{A}\cap \overline{B}=\O$ according to Lemma \ref{l:3.1}. Hence, $U(A,B)\ne\O$, that is there exists $x_0\in U(A,B)$. Now since $A_0\subseteq A$ and $B_0\subseteq B$, $x_0\in U(A_0,B_0)=X_0$. On other hand, using that $A_{n_k}\subseteq A$ and $B_{n_k}\subseteq B$ for every $k\in\mathbb N$, we obtain that 
$$
x_0\in \bigcap\limits_{k=1}^\infty U(A_{n_k},B_{n_k})\subseteq \bigcap\limits_{k=1}^\infty (X\setminus X_{n_k})=X\setminus
\left(\bigcup\limits_{k=1}^\infty X_{n_k}\right)=X\setminus X_0.
$$
This gives a contradiction

$(iii)\Rightarrow(iv)$. Let $(E_n)_{n=1}^\infty$ be a sequence of finite pairwise disjoint sets $E_n\subseteq Y$, which weakly converges to $y_0\in Y$. Let $$E=\bigcup\limits_{n=1}^\infty E_n=\{y_n:n\in {\mathbb N}\}.$$ 
Using the induction on $k$ it easy to construct a strictly decreasing sequence of infinite sets $N_k\subseteq {\mathbb N}$ such that for every $k\in\mathbb N$
at least one of the following conditions

($a$)\,\, $y_k\not\in \overline{\bigcup\limits_{n\in N_k}E_n}$;

($b$)\,\,the sequence $(E_n)_{n\in N_k}$ weakly converges to $y_k$;

\noindent holds.

We take a strictly increasing sequence $(n_k)_{k=1}^{\infty}$ of integers $n_k\in N_k$. For every $k\in\mathbb N$ we put 
$$
A_k=\{y_m\in E_{n_k}: \,\,\mbox{sequence}\,\, (E_n)_{n\in
N_m} \,\,\mbox{weakly}\,\,\mbox{converges\,\,to}\,\, y_m\}.$$ We show that there exists an integer $k_0$ such that $A_k\ne\O$ for every $k\ge k_0$.

Suppose that there exists an infinite set $M\subseteq \mathbb N$ such that $A_k=\O$ for every $k\in M$. This means that the condition ($a$) holds for every $k\in M$ and $y_m\in E_{n_k}$. Using that $n_i\in N_m$ for all $i\ge m$, we obtain that $y_m\not\in \overline{\bigcup\limits_{i\ge m}E_{n_i}}$. Therefore the set $\bigcup\limits_{k\in M}E_{n_k}$ is discrete. Using infinite subsets $M_1$ and $M_2$ of $M$ such that $M=M_1\sqcup M_2$, according to Corollary \ref{c:3.2}, we obtain that 
$$\left(\overline{\bigcup\limits_{k\in M_1}E_{n_k}}\right)\cap\left(\overline{\bigcup\limits_{k\in M_2}E_{n_k}}\right)=\O.$$
But this contradicts to the fact that the sequence $(E_{n_k})_{k\in M}$ weakly converges to $y_0$.

Now we show that the sequence $(A_k)_{k=1}^\infty$ weakly converges to every point $y\in \bigcup\limits_{k=1}^\infty A_k$.

Let $y_m\in \bigcup\limits_{k=1}^\infty A_k$. Suppose that $(A_k)_{k=1}^{\infty}$ does not weakly converge to $y_m$. Then there exists an infinite set set  $M\subseteq\mathbb N$ such that $y_m\not\in \overline{\bigcup\limits_{k\in M}A_k}$. Without loss of the generality we can propose that $\{n_k:k\in M\}\subseteq N_{m}$.
Note that as in the previous reasoning the set $\bigcup\limits_{k\in M}(E_{n_k}\setminus A_k)$ is discrete. Therefore, using Corollary \ref{c:3.2} we obtain that there exists an infinite set $M_1\subseteq M$ such that $y_m\not\in\overline{\bigcup\limits_{k\in M_1}(E_{n_k}\setminus A_k)}$. Thus, $y_m\not\in \overline{\bigcup\limits_{k\in M_1}E_{n_k}}$. This contradicts to the fact that the sequence $(E_{n_k})_{k\in M_1}$ weakly converges to $y_m$.

$(iii)\Rightarrow(i)$. Let a sequence $(E_n)_{n=1}^{\infty}$ of nonempty finite pairwise disjoint sets weakly converges to a point $y_0\in Y$. For every $n\in\mathbb N$
we put $G_n=\bigcup\limits_{k\ge n}U(E_k,E_{k+1})$. It easy to see that all sets $G_n$ are open and everywhere dense in $X$. Therefore the sets $F_n=X\setminus G_n$ are nowhere dense in $X$. Now it is sufficient to prove that $\bigcap\limits_{n=1}^\infty G_n=\O$.

Assume that $x_0\in\bigcap\limits_{n=1}^\infty G_n$. Then there exists a strictly increasing sequence $(k_n)_{n=1}^\infty$ of integers $k_n\in\mathbb N$ such that $x_0\in U(E_{k_n},E_{k_n+1})$ for every $n\in\mathbb N$, that is $x_0(y)=0$ for every $y\in \bigcup\limits_{n=1}^\infty E_{k_n}$ and $x_0(y)=1$ for every $y\in \bigcup\limits_{n=1}^\infty E_{k_n+1}$. Since $(E_n)_{n=1}^{\infty}$ weakly converges to $y_0$ in $Y$, the oscillation of the function $x_0$ on each neighborhood $V$ of $y_0$ equals to 1. But this contradicts to the continuity of $x_0$ at $y_0$.
\end{proof}

\bibliographystyle{amsplain}

\end{document}